\documentclass[12pt, oneside, reqno]{amsart}
\usepackage{amsmath}
\usepackage{graphicx}
\usepackage{hyperref}
\usepackage[latin1]{inputenc}
\usepackage{mathtools}
\usepackage{amsfonts}
\usepackage{amssymb}
\usepackage[T1]{fontenc}
\usepackage{amsthm}
\usepackage{fullpage}
\usepackage{xcolor}
\usepackage{longtable}
\usepackage{float}

\newtheorem{theorem}{Theorem}[section]
\newtheorem{corollary}[theorem]{Corollary}
\newtheorem{lemma}[theorem]{Lemma}
\newtheorem{Conjecture}[theorem]{Conjecture}

\theoremstyle{definition}
\newtheorem{definition}{Definition}[section]
 
\theoremstyle{remark}

\numberwithin{equation}{section}

\newcommand{\F}{\mathbb{F}_q}
\newcommand{\Fm}{\mathbb{F}_{q^m}}
\newcommand{\B}{\mathfrak{B}}
\newcommand{\C}{\mathfrak{C}}

\newcommand{\Sc}{\Fm\setminus \bigcup_{i=1}^m A_i}

\title{Existence of Special Types Primitive Pairs in Finite Fields Avoiding Affine Hyperplanes }

\author{Himangshu Hazarika}
\address{Department of Mathematics, 
Tezpur College, Tezpur-784001,  Assam, India}
\email{diku\_95@tezu.ernet.in}

\author{Giorgos Kapetanakis}
\address{Department of Mathematics, University of Thessaly, 3rd km Old National Road Lamia-Athens, 35100 Lamia, Greece}
\email{gnkapet@gmail.com}

\author{Dhiren Kumar Basnet*}
\address{Department of Mathematical Sciences, Tezpur University-784128, Assam, India}
\email{dbasnet@tezu.ernet.in}

\thanks{}

\begin{document}

\begin{abstract}
Let $\Fm$ be finite fields of order $q^m$, where $m\geq 2$ and $q$, a prime power. Given $\F$-affine hyperplanes $A_1,\ldots, A_m$ of $\Fm$ in general position, we study the existence of primitive element  $\alpha$ of $\Fm$, such that  $f(\alpha)$  is also primitive, where $ax^2+bx+c\in \Fm[x]$ ($a\neq 0$ and $b^2\neq 4ac$)  in $\Fm$ and the primitive pair $(\alpha, f(\alpha))$ avoids each $A_i$. We establish results for fields of higher order.
 \end{abstract}

\maketitle

\section{Introduction}

For an integer $m\geq 2$ and a prime power $q$, the finite field of order $q$ is denoted by $\F$ and its extension of order $q^m$ is denoted by $\Fm$. It is well-known that the multiplicative group $\Fm^*$ of the finite field $\Fm$ is cyclic and a generator of this cyclic group is called a \emph{primitive element}.

The existence of primitive elements is well-known for arbitrary finite fields and is a direct consequence of the fact that any finite subgroup of a field's multiplicative group is cyclic. In addition, in \cite{3}, the existence of a primitive element $\alpha$ of $\Fm$ such that $f(\alpha)= a\alpha^2 + b\alpha +c$  (where $a\neq 0$ and $b^2\neq 4ac$) is also primitive of $\Fm$ was proven, given that $q^m\geq 211$. 

%
\begin{theorem}[Booker-Cohen-Sutherland-Trudgian] \label{thm:bcst}
For $q\geq 211$, there always exists a primitive element in $\alpha$ in $\F$ such that $g(\alpha)$ is also a primitive element of $\F$, where $g(x)= ax^2+bx+c \in \F[x]$, such that $a\neq 0$ and $b^2\neq 4ac$.
\end{theorem}

 On a different note, the extended field $\Fm$ of $\F$ of degree $m$ can be treated as vector space $\Fm(\F)$ of dimension $m$. 
We assume that $\B = \{\beta_1, \ldots , \beta_m \}$ be a basis of the vector space $\Fm(\F)$.  Then each element $\beta$ of the vector space can be expressed uniquely as follows
\begin{equation*}
\beta =\sum_{i=1}^m a_i\beta_i, \, \, a_i\in \F.
\end{equation*}
Take some $c_1, \ldots, c_m \in \F$,
which we will treat as arbitrary yet fixed for the rest of this work. Then each set

\begin{equation*}
A_j =\left \{ \sum_{i=1}^m a_i\beta_i :a_i\in \F,  a_j=c_j \right\}
\end{equation*}
determines an affine hyperplane of $\Fm$. Furthermore, the set $\C = \{ A_1, \ldots, A_m \}$ is a set of \emph{$\F$-affine hyperplanes in general positions}, and, for $1 \leq k \leq m$, the intersection of any $k$ distinct elements of $\C$ is an $\F$-affine space of dimension $m-k$. Also, for the sake of simplicity, throughout this paper, we set $S_c^*:= \Sc$.

 In \cite{1}, the following theorem was established.

\begin{theorem}[Fernandes-Reis] \label{thm:fr}
Take $m\geq 2$ a positive integer and let $\C = \{ A_1, \ldots, A_m \}$ be a set of $\F$-affine hyperplanes of $\Fm$ in general position. Then there exists a primitive element in $S_c^*$ of $\Fm$  provided one of the following holds
\begin{enumerate}
\item  $q\geq 11$, except  $(q, \,m)$ is one of  the pairs $(11,  4)$,  $(11,  6)$,  $(11,  12)$ and $(13, 4)$;
\item  $q=7, \,8, \,9$, and $m$ is large enough.
\end{enumerate}
\end{theorem}

An extension of Theorem~\ref{thm:fr}, for $q=4, \, 5$, was given in \cite{2}, by using modified bounds. 

\begin{theorem}[Grzywaczyk-Winterhof] \label{thm:gw}
Let $m\geq 2$ be a positive integer and  $\C = \{ A_1, \ldots, A_m \}$ be a set of $\F$-affine hyperplanes of $\Fm$ in general position. Then there exists a primitive element in $S_c^*$ of $\Fm$ for $q=4,\, 5$ and $m$ is large enough.
\end{theorem}
It is worth mentioning that, in the same paper, the authors outlined possible treatments for the cases $q=3$ and $q = 2$. However, the latter case appears to be trivial. For more details see \cite{2}.



 Motivated by the aforementioned works,  here we take $f(x)= ax^2 + bx +c\in\F[x]$, where $a\neq 0$ and $b^2\neq 4ac$ and explore the existence of $\alpha\in\Fm$, such that both $\alpha$ and $f(\alpha)$ are primitive and belong to $S_c^*$; such a pair $(\alpha, f(\alpha))$ is called \emph{primitive $S_c^*$ pair}. The fruits of our efforts are summarized in the following theorems.
%
%
%

\begin{theorem}\label{main1}
Let $m\geq 5$, be a positive integer and  $C= \{ A_1, \ldots, A_m \}$  a set of $\F$-affine hyperplanes of $\Fm$ in general position. Then the set $S_c^*= \Sc$ contains a primitive $S_c^*$ pair $(\alpha, f(\alpha))$ of $\Fm$, provided one of the following holds:
\begin{enumerate}
\item $q\geq 25$, except when $(q,m)$ is one of $(25, 5)$,  $(25, 6)$,  $(29, 6)$ and $(31, 6)$;
\item $q=19$ and $m\neq 5, 6, 8$;
\item $q=17$ and $m \neq 6, 8$;
\item $q=16$ and $m\neq 5, 6, 7, 9$;
\item $q=11, 13$ and $m \neq 5, 6, 7, 8, 9, 10, 12$; and $(11, 14)$,  $(11, 15)$,
\item $q=3, 4, 5, 7, 8, 9$ and $m$ is large enough.
\end{enumerate}
\end{theorem}

For $m=2,\, 3, \, 4$, we combine asymptotic results with refined results based on calculations and obtain the following.

\begin{theorem}\label{main2}
Let $C= \{ A_1, \ldots, A_m \}$ be a set of $\F$-affine hyperplanes of $\Fm$ in general position. Then the set $S_c^*= \Sc$ contains a primitive $S_c^*$ pair $(\alpha, f(\alpha))$ of $\Fm$, provided one of the following holds:
\begin{enumerate}
\item $m=2, \, q> 9.4718 \times 10^{13}$, 
\item $m=3, \, q> 6.601 \times 10^{11}$,
\item $m=4, \, q> 1.271 \times 10^{8}$.
\end{enumerate}
\end{theorem}
%
%

The outline of this paper is as follows: In Section~\ref{sec2},  we present some background material on characters; in particular, we present the character sum estimates we will need and establish a novel incomplete character sum estimate. In Section~\ref{sec3}, we prove two nontrivial bounds of incomplete character sums that run through the set $S_c^*$ in $\Fm$, which are required in proceeding sections. Then, in Section~\ref{sec4},  we establish a sufficient condition for the existence of primitive $S_c^*$ pair $(\alpha, f(\alpha))$ in $\Fm$. Additionally we provide some asymptotic results by using the above condition. In Section~\ref{sec5} we employ the ``modified prime sieve technique'' to develop a sufficient condition for more efficient results
and provide concrete results. 
Finally, in Section~6, we combine Theorem~\ref{main2} with extensive computational evidence in an attempt to clarify the situation for the cases $2\leq m\leq 4$, leading to Conjecture~\ref{con1}.

\section{Preliminaries} \label{sec2}

A well-known expression of the characteristic function for the primitive elements of $\F$ is given in \cite{7} as
\begin{equation} \label{Vino1}
 \Gamma_{\F}(\lambda)=\theta(q-1)\sum_{d\mid q-1}\left(\frac{\mu(d)}{\phi(d)}\sum_{\chi_d\in \widehat{\F^*}}  \chi_d(\lambda)\right),
\end{equation}
     where $\theta(e):=\frac{\phi(e)}{e}$, $\phi$ is the Euler's totient function, $\mu$ is the M\"obius function and $\chi_d$ stands for any multiplicative character of $\F$ of order $d$.
     
     Now, take a set $A\subseteq \F^*$, the number $P(A)$ of primitive elements of $\F$ in $A$ is given by
     \[ P(A) = \sum_{\lambda\in A} \Gamma_{\F}(\lambda) , \]
which, combined with \eqref{Vino1}, yields
%
\begin{equation} \label{Vino}
 P(A)=\theta(q-1)\sum_{d\mid q-1}\left(\frac{\mu(d)}{\phi(d)}\sum_{\chi_d\in \widehat{\F^*}} \sum_{\lambda \in A}    \chi_d(\lambda)\right).
\end{equation}
%
%
%
%
\begin{definition}
For $e\mid q^m-1$, an element $\lambda$ of $\F^*$ is called \emph{$e$-free} if $d\mid e$ and $\lambda= \gamma^d$, for some $\gamma \in \Fm$, imply $d=1$. Furthermore an element $\lambda$ of $\Fm$ is primitive if and only if it is a $(q^m-1)$-free element.
\end{definition}

For any $e\mid q^m-1$, from the results by Cohen and Huczynska~\cite{4,5}, the characteristic function for the set of $e$-free elements is as follows:
\begin{equation*}
\rho_e: \lambda\mapsto\theta(e)\sum_{d\mid e}\left(\frac{\mu(d)}{\phi(d)}\sum_{\chi_d}\chi_d(\lambda)\right).
\end{equation*}
%
%
In the following sections, we will encounter various character sums and some estimations, for they will be necessary. The following lemmas are well-known and provide such results.

\begin{theorem}[{\cite[Theorem~5.41]{7}}] \label{charbound1}
Let $\chi$ be a nonprincipal multiplicative character of $\Fm$ of order $r>1$, and let $g(x) \in \Fm[x]$ be a monic polynomial of positive degree that is not an $r$-th power of a polynomial. If $g$ has $d$ distinct roots in its splitting field, then for every $a\in \F$, we have that
\begin{equation*}
\left| \sum_{\lambda \in \Fm} \chi(ag(\lambda))  \right| \leq (d-1)q^{m/2}.
\end{equation*} 
\end{theorem}

\begin{theorem}[{\cite[Corollary~2.3]{8}}] \label{charbound2}
Consider two nontrivial multiplicative characters $\chi_1,\chi_2$ of $\Fm$. Let $g_1(x)$ and $g_2(x)$ be two monic co-prime polynomials in $\Fm[x]$, such that none of $g_i(x)$ is of the form $f(x)^{ord(\chi_i)}$ for $i=1,2$; where $f(x)\in \Fm[x]$ with degree at least 1. Then
\[ \left| \sum_{\lambda\in \Fm}\chi_1(g_1(\lambda))\chi_2(g_2(\lambda))\right| \leq (k_1+k_2-1)q^{m/2}, \]
 where $k_1$ and $k_2$ are the degrees of largest square free divisors of $g_1$ and $g_2$ respectively. 
\end{theorem}

%
%

We will also need the following results.
\begin{theorem}[{\cite[Theorem~3.1]{1}}] \label{Reis}
Let $\C = \{A_1, \ldots, A_m\}$ be a set of $\F$-affine hyperplanes of $\Fm$ in general position and $\chi$ be a nonprincipal multiplicative character over $\Fm$. Then
\begin{equation*}
s(S_c^*, \, \chi) \leq \delta(q, m) \leq (2^m-1)q^{m/2},
\end{equation*}
where, $s(S_c^*, \, \chi) := \left| \sum_{\lambda \in S_c^*} \chi(\lambda)\right|$, $S_c^*= \Sc$ and $\delta(q, r) = \sum_{i=0}^{m-1} \binom{k}{i} q^{min\{ i, \frac{m}{2}\}}$.
\end{theorem}

\begin{lemma}[{\cite[Lemma~2.2]{1}}]\label{logbound}
For a positive integer $n$, let $W(n)$ denote the number of squarefree divisors of $n$. Then for $n\geq 3$, 
\begin{equation*}
W(n-1) < n^{\frac{0.96}{\log \log n}}.
\end{equation*}
\end{lemma}
Last but not the least, we will need to estimate an incomplete character sum involving a multiplicative character with polynomial argument, that runs through an affine hyperplane. In order to achieve this, inspired by the work of Reis~\cite{reis21}, we will first need the following lemma that is a direct consequence of \cite[Lemma~3.4]{reis21}.
\begin{lemma} \label{lemma:reis}
Let $A$ be an affine hyperplane of $\Fm$. There exists a separable polynomial $h_A\in\F[x]$ of degree $q^{m-1}$ whose derivative is a constant and satisfies the following:
\begin{enumerate}
  \item $h_A(\Fm) = A$;
  \item for every $a\in A$,  the equation $h_A(x) = a$ has exactly $q^{m-1}$ solutions in $\Fm$.
\end{enumerate}
\end{lemma}
Now, we can prove the character sum estimate we will need, which generalizes \cite[Corollary~3.5]{reis21}.
\begin{corollary} \label{cor:incomplete}
  Let $\chi$ be a nonprincipal multiplicative character of $\Fm$, $A$ an affine hyperplane of $\Fm$, $g(x)\in\Fm[x]$ a monic polynomial of positive degree that is not an $m$-th power of a polynomial. If $g$ has $d$ distinct roots in its splitting field, then for every $a\in \F$, we have that
\[
\left| \sum_{\lambda\in A} \chi(ag(\lambda)) \right| \leq \frac{dq^{m-1}-1}{q^{m-1}} \cdot q^{m/2} .
\]
\end{corollary}
\begin{proof}
Lemma~\ref{lemma:reis} implies that there exists a separable polynomial $h_A\in\F[x]$ of degree $q^{m-1}$ of constant derivative, such that $h_A(\Fm) = A$ and that for every $a\in A$, the equation $h_A(x) = a$ has exactly $q^{m-1}$ solutions in $\Fm$. It follows that
\begin{equation}\label{ip2}
\sum_{\lambda\in A} \chi(ag(\lambda)) = \frac{1}{q^{m-1}} \sum_{\lambda\in \Fm} \chi (a g(h_A(\lambda))) .
\end{equation}
Clearly, $\lambda$ is a root of $g(h_A(x))$ if and only if $h_A(\lambda)=a$, where $a$ is a root of $g$. Since $g$ has $d$ roots and each of the equations of the form $h_A(\lambda)=a$ has $q^{m-1}$ solutions, we obtain that $g(h_A(x))$ has $dq^{m-1}$ roots in its splitting field.

Furthermore, the polynomial $g(h_A(x))$ is not an $m$-th power of a polynomial. In order to make this clear, observe that $g(h_A(x))$ is an $m$-th power if and only if the multiplicity of each of its roots (in its splitting field) is a multiple of $m$. Recall that $g$ is not an $m$-th power and $h_A$ is separable of constant derivative. Now, take a root $y$ of $g\circ h_A$. Then $g(h_A(y)) = 0\iff h_A(y) = z \iff h_A(y)-z=0$, where $z$ is a root of $g$. However, given that $h_A$ has constant derivative, the polynomial $h_A(x)-z$ only has simple roots, that is, the multiplicity of $y$ as a root of $g\circ h_A$ coincides with the multiplicity of $z$ as a root of $g$. It follows that, since $g$ is not an $m$-th power, neither is $g\circ h_A$.

Now, Theorem~\ref{charbound1} implies
\[ \left| \sum_{\lambda\in \Fm} \chi (a g(h_A(\lambda))) \right| \leq (dq^{m-1}-1) q^{m/2} , \]
which combined with \eqref{ip2} completes the proof.
\end{proof}
\section{Lower bounds on elements of $\Fm$ avoiding affine hyperplanes} \label{sec3}

We provide two theorems for lower bounds on elements of $\Fm$ avoiding affine hyperplanes, and we prove only the second theorem in this paper, since the proof of the first one is similar. 
Throughout this paper, we use the notation $Q_{n, e}$ for the set of rational functions $g(x) \in \Fm(x)$ of degree at most $n$, which are not $e$-th power of any rational function in $\overline{\mathbb{F}_p(x)}$, and $\xi$ be the set of exceptional functions $g(x) \in \Fm(x)$. Note that, $g(x)\in\Fm(x)$ is an \emph{exceptional function} if it can be written as $g(x)= \alpha((h(x)^p-h(x)) + \beta x$, for some $h(x) \in \overline{\mathbb{F}_p(x)}$, and $\alpha, \beta \in \overline{\mathbb{F}_p}$.

\begin{theorem}\label{af1}
Let $\chi$ be a nonprincipal multiplicative character and  $g(x) \in \Fm(x)$ be a rational function of degree at most $d$. We assume that $\chi$ is of order $e$ and $g(x) \in Q_{d,e}$, then for the set $S_c^*=\Sc$, where $A_i$'s are affine hyperplanes,  we have

\begin{equation*}
 \left| \sum_{\lambda \in S_c^*} \chi (g(\lambda))\right| \leq k 2^m q^{m/2},
\end{equation*}
where $k= \deg (g(x))$.
\end{theorem}

\begin{theorem}\label{af2}
Let $\chi_1,\chi_2$ be nonprincipal multiplicative characters of $\Fm$, and $g_1(x), g_2(x)\in \Fm[x]$ be two coprime polynomials of degree at most $d$, such that 
\begin{enumerate}
\item $g_1(x) \in Q_{d, \omega_1}$ and $g_2(x) \in Q_{d, \omega_2} $, where $\omega_i= ord(\chi_i)$ for $i$= 1, 2,
\item if $\chi_1(g_1(x)) = \chi_{q^m-1}(g_1(x)^{n_1})$ and $\chi_2(g_2(x)) = \chi_{q^m-1}(g_2(x)^{n_2})$, 
then $g_1(x)^{n_1}g_2(x)^{n_2} \in Q_{d,q^m-1}$,
\end{enumerate}
then  for the set $S_c^*=\Sc$, where $A_i$'s are affine hyperplanes,  we have
\begin{equation*}
 \left| \sum_{\lambda \in S_c^*} \chi_1(g_1(\lambda))\chi_2(g_2(\lambda))\right| < (k_1+k_2)2^mq^{m/2},
\end{equation*}
where $k_i= \deg (g_i(x))$ for $i=1, 2$.
\end{theorem}
\begin{proof}
For $\lambda \in S_c^*= \Sc$, where $A_i$'s are affine hyperplanes,  we have
\begin{align*}
\left|  \sum_{\lambda \in S_c^*} \chi_1(g_1(\lambda))\chi_2(g_2(\lambda))\right|  =& \left| \sum_{\lambda \in \Fm} \chi_1(g_1(\lambda))\chi_2(g_2(\lambda)) \Gamma_{S_c^*}(\lambda)\right| \\
=& \left|  \sum_{\lambda \in \Fm} \chi_1(g_1(\lambda))\chi_2(g_2(\lambda))  \prod_{i=1}^r \left(1- \Gamma_{A_i}(\lambda)\right) \right|\\
=& \left|  \sum_{\lambda \in \Fm} \chi_1(g_1(\lambda))\chi_2(g_2(\lambda)) \left( 1+ \sum_{i=1}^r \sum_{\substack{J\subseteq [1,m] \\ \#J=i}} (-1)^i \Gamma_{A_J}(\lambda)\right) \right|\\
 \leq &  \left|  \sum_{\lambda \in \Fm} \chi_1(g_1(\lambda))\chi_2(g_2(\lambda))\right| \\ & + \left|  \sum_{\lambda \in \Fm} \chi_1(g_1(\lambda))\chi_2(g_2(\lambda))  \sum_{i=1}^r \sum_{\substack{J\subseteq [1,m] \\ \#J=i}} (-1)^i \Gamma_{A_J}(\lambda)  \right|.\\
\end{align*}
By Theorem~\ref{charbound2}, $\left|  \sum_{\lambda \in \Fm} \chi_1(g_1(\lambda))\chi_2(g_2(\lambda))\right|\leq (k_1+k_2-1)q^{m/2}$, thus
\begin{align*}
\left|  \sum_{\lambda \in S_c^*} \chi_1(g_1(\lambda))\chi_2(g_2(\lambda))\right|  
 \leq &   (k_1+k_2-1)q^{m/2}  +    \sum_{i=1}^r \sum_{\substack{J \subseteq[1,m] \\ \#J=i}} \left| \sum_{\lambda \in \Fm} \chi_1(g_1(\lambda))\chi_2(g_2(\lambda))  \Gamma_{A_J}(\lambda)  \right|\\
= &   (k_1+k_2-1)q^{m/2}  +  \sum_{i=1}^r \sum_{\substack{J \subseteq[1,m] \\ \#J=i}} \left| \sum_{\lambda \in A_J} \chi_1(g_1(\lambda))\chi_2(g_2(\lambda))    \right|\\
= &   (k_1+k_2-1)q^{m/2}  +  \sum_{i=1}^r \sum_{\substack{J \subseteq[1,m] \\ \#J=i}} \left| \sum_{\lambda \in A_J} \chi_{q-1}(g_1(\lambda)^{n_1}g_2(\lambda)^{n_2})    \right|.\\
\end{align*}
Then by Corollary~\ref{cor:incomplete},  
\[ \left| \sum_{\lambda \in A_J} \chi_{q-1}(g_1(\lambda)^{n_1}g_2(\lambda)^{n_2})    \right| \leq \frac{(k_1+k_2)q^{m-1}-1}{q^{m-1}} q^{m/2} < (k_1+k_2)q^{m/2} , \]
hence, 
\begin{align*}
\left|  \sum_{\lambda \in S_c^*} \chi_1(g_1(\lambda))\chi_2(g_2(\lambda))\right|  
 < &   (k_1+k_2-1)q^{m/2}  +    \sum_{i=1}^r \sum_{\substack{J \subseteq[1,m] \\ \#J=i}} (k_1+k_2) q^{m/2}\\
 < &   (k_1+k_2)q^{m/2}  +  (2^m-1)(k_1+k_2)q^{m/2} .
\end{align*}
This completes the proof. 
\end{proof}

\section{Sufficient existence condition and some asymptotic results}\label{sec4}

Let $f(x)= ax^2+bx+c \in \Fm[x]$ be such that $a\neq 0$ and $b^2\neq 4ac$. In this section, we aim to assess the number of primitive elements $\alpha$ of $\Fm$ such that $f(\alpha)$ is also primitive in $\Fm$ and both $\alpha, f(\alpha)$ belong to $S_c^*= \Sc$, i.e., they avoid the collection of $\F$-affine hyperplanes $\C= \{ A_1, \ldots, A_m\}$ in general position in $\Fm$. We denote by $\Lambda(d)$,  the set of multiplicative characters of $\Fm^*$ of order $d$. 
Then we have the following theorem.

\begin{theorem} \label{thm4.1}
Let $P(S_c^*)$ be the number of \emph{primitive $S_c^*$ pairs}  $ (\alpha, f(\alpha))$ in $S_c^*=\Sc$   avoiding the collection $\C= \{ A_1, \ldots, A_m\}$ of $\F$-affine hyperplanes in general position in $\Fm$, where $f(x)= ax^2+bx+c \in \Fm[x]$, $a\neq 0$ and $b^2\neq 4ac$. Then a sufficient condition for 
$P(S_c^*)>0$ is
\begin{equation*}
\left( \frac{q-1}{2\sqrt{q}} \right)^m > 3W(q^m-1)(1+ W(q^m-1)).
\end{equation*}
\end{theorem}

\begin{proof}
>From \eqref{Vino}, we have 
\begin{align*}
P(S_c^*)= \theta(q^m-1)^2 \sum_{d_1,d_2\mid q^m-1}\frac{\mu(d_1)\mu(d_2)}{\phi(d_1)\phi(d_2)}\sum_{\substack{\chi_{d_i}\in \Lambda(d_i) \\ i=1,2}} \sum_{\lambda \in S_c^*}\chi_{d_1}(\lambda)\chi_{d_2}(f(\lambda)).
\end{align*}
Therefore,
\begin{align*}
\frac{1}{\theta(q^m-1)^2}P(S_c^*)& = \sum_{\lambda \in S_c^*} \chi_{1}(\lambda)\chi_{1}(f(\lambda)) + \sum_{1\neq d_1\mid q^m-1} \frac{\mu(d_1)}{\phi(d_1)} \sum_{\chi_{d_1}\in \Lambda(d_1)} \sum_{\lambda \in S_c^*} \chi_{d_1}(\lambda)\\
&+ \sum_{1 \neq d_2\mid q^m-1}\frac{\mu(d_2)}{\phi(d_2)} \sum_{\chi_{d_2}\in \Lambda(d_2)} \sum_{\lambda \in S_c^*} \chi_{d_2}(f(\lambda))\\
& + \sum_{\substack{ d_1,d_2\mid q^m-1 \\ d_1,d_2\neq 1}}\frac{\mu(d_1)\mu(d_2)}{\phi(d_1)\phi(d_2)} \sum_{\substack{\chi_{d_i}\in \Lambda(d_i) \\ i=1,2}} \sum_{\lambda \in S_c^*}\chi_{d_1}(\lambda)\chi_{d_2}(f(\lambda)).
\end{align*}
Then upon applying the triangular inequality we obtain
\begin{multline} \label{ip1}
\frac{1}{\theta(q^m-1)^2} P(S_c^*) \geq \left| \sum_{\lambda \in S_c^*} 1\right| - \sum_{1\neq d_1\mid q^m-1} \frac{1}{\phi(d_1)} \sum_{\chi_{d_1}\in \Lambda(d_1)} \left| \sum_{\lambda \in S_c^*} \chi_{d_1}(\lambda)\right| \\
  - \sum_{1 \neq d_2\mid q^m-1} \frac{1}{\phi(d_2)} \sum_{\chi_{d_2}\in \Lambda(d_2)} \left|\sum_{\lambda \in S_c^*} \chi_{d_2}(f(\lambda))\right|\\
  - \sum_{\substack{1\neq d_i\mid q^m-1 \\ i=1,2}} \frac{1}{\phi(d_1)\phi(d_2)} \sum_{\substack{\chi_{d_i}\in \Lambda(d_i) \\ i=1,2}} \left| \sum_{\lambda \in S_c^*} \chi_{d_1}(\lambda)\chi_{d_2}(f(\lambda)) \right|.
\end{multline}
Theorem~\ref{Reis} implies
\begin{equation*}
 \left|\sum_{\lambda \in S_c^*}\chi_{d_1}(\lambda)\right| \leq (2^m-1)q^{m/2} , 
\end{equation*}
Theorem~\ref{af1} implies
\begin{equation*}
\left|\sum_{\lambda \in S_c^*}\chi_{d_2}(f(\lambda))\right| \leq 2.2^m q^{m/2} = 2^{m+1}q^{m/2},
\end{equation*}
and, should the conditions of Theorem~\ref{af2} be met, then
\begin{equation*}
\left|\sum_{\lambda \in S_c^*}\chi_{d_1}(\lambda)\chi_{d_2}(f(\lambda)) \right|\leq (2+1)2^m q^{m/2}= 3.2^mq^{m/2}.
\end{equation*}
For the latter $d_1, d_2 \neq1$, hence there exist some $n_1, n_2\in \{0,1, \ldots, q^m-2\}$ such that $\chi_{d_1}(\lambda)\chi_{d_2}(f(\lambda))=\chi_{q^m-1}(\lambda^{n_1}f(\lambda)^{n_2})  $. 
Now we show that $x^{n_1}(ax^2+bx+c)^{n_2}\in Q_{d,q^m-1}$.

Suppose that $x^{n_1}(ax^2+bx+c)^{n_2}\not\in Q_{d,q^m-1}$, i.e. $x^{n_1}(ax^2+bx+c)^{n_2}=yH(x)^{q^m-1}$, for some $y\in\F$ and $H(x)\in\Fm[x]$. 
If $n_1\neq 0$ then $x^{n_1}\mid H(x)^{q^m-1}$, hence $(ax^2+bx+c)^{n_2}= y x^{q^m-1-n_1}B(x)^{q^m-1}$,  where $B(x)= \frac{H(x)}{x}\in \Fm[x]$.
Thus $x^{q^m-1-n_1}|(ax^2+bx+c)^{n_2}$, which is possible only if $c=0$.
 Hence $x^{n_1}(ax^2+bx)^{n_2}=y H(x)^{q^m-1}$, i.e. $x^{n_1+n_2}(ax+b)^{n_2}=y H(x)^{q^m-1}$. Recall that $a\neq 0$ and observe that the condition $b^2\neq 4ac$ yields $b\neq 0$, thus $(ax+b)^{n_2}|B(x)^{q^m-1}$. This is possible only if $n_2=0$, which yields $x^{n_1} = yH(x)^{q^m-1}$, a contradiction. It follows that $n_1=0$.
%
%

Now $n_1=0$ implies $(ax^2+bx+c)^{n_2}= yH(x)^{q^m-1}$, i.e., $2n_2\geq K(q^m-1)$, where $K$ is the degree of $H(x)$. 
If $n_2\neq 0$, then $2n_2\geq K(q^m-1)$ implies $K\leq 1$.
 If $K=0$, then $(ax^2+bx+c)^{n_2}= y$ implies $n_2=0$, thus we are done. 
If $K=1$, then $(ax^2+bx+c)^{n_2}= y(a^\prime x+b^\prime)^{q^m-1}$. So, $n_2\neq 0$, hence $ax^2+bx+c= (a^{\prime\prime}x+ b^{\prime\prime})^2$, which is impossible. Thus $n_2=0$ is the only possible case.
It follows that $n_1=0$ and $n_2=0$, i.e., $H(x)$ is a constant, and $x^{n_1}(ax^2+bx+c)^{n_2} \in Q_{d,q^m-1}$.

Then considering the above inequalities, \eqref{ip1} yields
 
\begin{multline*}
\frac{1}{\theta(q^m-1)^2}P(S_c^*) \geq (q-1)^m - \sum_{1\neq d_1\mid q^m-1} \frac{1}{\phi(d_1)} \sum_{\chi_{d_1}\in \Lambda(d_1)}  (2^m-1)q^{m/2}
 \\ - \sum_{1\neq d_2\mid q^m-1} \frac{1}{\phi(d_2)} \sum_{\chi_{d_2}\in \Lambda(d_2)} 2^{m+1}q^{m/2} - \sum_{\substack{1\neq d_i\mid q^m-1 \\ i=1,2}} \frac{1}{\phi(d_1)\phi(d_2)} \sum_{\substack{\chi_{d_i}\in \Lambda(d_i) \\ i=1,2}} 3.2^mq^{m/2}.
\end{multline*}
Thus,
\begin{multline} \label{eq4.1}
\frac{P(S_c^*)}{\theta(q^m-1)^2} \geq (q-1)^m -  (2^m-1)q^{m/2} W(q^m-1)
 \\ -  2^{m+1}q^{m/2} W(q^m-1)  
 -3.2^{m+1}q^{m/2} W(q^m-1)^2
\end{multline}
and the desired result follows.
\end{proof}

We apply Lemma~\ref{logbound} on the sufficient condition and obtain some asymptotic results.  From \cite{1,2}, it is clear that in $\Fm$, $m$ should be large enough for the fields where $4\leq q\leq 9$, for the existence of a primitive element of $\Fm$ in $S_c^*=\Sc$.
Also, it is clear that when $q=2$, no such primitive element exists.
Hence we focus on $q\geq 11$. 
Furthermore we provide asymptotic results for $m=2,\, 3, \, 4$ and concrete results for $m\geq 5$. 
In particular, a direct application of Theorem~\ref{thm4.1}, combined with Lemma~\ref{logbound} yield Theorem~\ref{main2}.

%

For $m\geq 5$, we have only 162 possible exceptional pairs. Which we will discuss in the next section by using the \emph{modified prime sieve technique}. 
These pairs are
$(11, 5)$, $(11, 6)$,  $(11, 7)$,  $(11, 8)$,  $(11, 9)$,   $(11, 10)$,  $(11, 12)$, $(11, 14)$,  $(11, 15)$,  $(11, 16)$, $(11, 18)$,   $(11, 21)$,  $(11, 22)$,  $(11, 24)$,  $(11, 26)$, $(11, 28)$,  $(11, 30)$,  $(11, 32)$,  $(11, 33)$,   $(11, 36)$,$(11, 39)$,  
  $(13, 5)$, $(13, 6)$,  $(13, 7)$,  $(13, 8)$,  $(13, 9)$,   $(13, 10)$,  $(13, 11)$, $(13, 12)$, $(13, 14)$,  $(13, 16)$,  $(13, 18)$,  $(13, 20)$,  $(13, 21)$, $(13, 22)$   $(13, 24)$,  $(13, 26)$, $(13, 30)$, $(13, 36)$,  $(13, 40)$, 
 $(16, 5)$,  $(16, 6)$,  $(16, 7)$,  $(16, 8)$,  $(16, 9)$,  $(16, 10)$,  $(16, 11)$,  $(16, 12)$,  $(16, 13)$,  $(16, 14)$,  $(16,15)$,  $(16, 18)$,  $(16, 21)$ , $(16, 22)$,  $(16, 24)$,  $(16, 25)$,  $(16, 27)$,  $(16, 30)$,  $(16, 36)$,  $(16,45)$, 
 $(17, 6)$, $(17, 8)$, $(17, 10)$, $(17, 12)$,  $(17, 16)$, $(17, 18)$, $(17, 20)$, $(17, 24)$,  $(17, 30)$, $(17, 36)$, 
$(19, 5)$, $(19, 6)$, $(19, 8)$, $(19, 10)$, $(19, 12)$,  $(19, 14)$, $(19, 15)$, $(19, 16)$, $(19, 18)$, $(19, 24)$,  $(19, 30)$, 
$(23, 5)$, $(23, 8)$,  $(23, 9)$,   $(23, 10)$, $(23, 12)$, $(23, 14)$, $(23, 16)$,  $(23, 18)$,   $(23, 20)$,  
$(25, 5)$, $(25, 6)$, $(25, 7)$, $(25, 8)$,  $(25, 9)$, $(25, 10)$, $(25, 12)$,  $(25, 15)$,  $(25, 18)$,
 $(27, 5)$, $(27, 6)$, $(27, 8)$, $(27, 10)$,  $(27, 12)$, $(27, 16)$,
  $(29, 6)$,  $(29, 8)$,  $(29, 10)$,  $(29, 12)$, 
 $(31, 6)$, $(31, 8)$,  $(31, 12)$,
 $(32, 6)$, $(32, 8)$,  $(32, 12)$, 
 $(37, 6)$,  $(37, 8)$,  $(37, 12)$, 
 $(41, 6)$,  $(41, 8)$,
  $(43, 6)$,  $(43, 8)$,   
  $(47, 6)$,  $(47, 8)$,  $(47, 12)$,        
 $(49, 6)$, 
 $(53, 6)$,
 $(59, 6)$,  $(59, 8)$, 
 $(61, 6)$,
 $(64, 6)$,  $(64, 8)$, $(64, 12)$,
 $(67, 6)$,
 $(79,6)$,
 $(81, 6)$,
 $(83, 6)$,  $(83, 8)$,
 $(89, 6)$,
 $(101, 6)$,
 $(103, 6)$,
 $(107, 6)$,
 $(109, 6)$,
 $(121, 6)$,
 $(131, 6)$,
 $(137, 6)$,
 $(139, 6)$,
 $(149, 6)$,
 $(179, 6)$,
 $(181, 6)$,
 $(221, 6)$, 
 $(229, 6)$,
 $(233, 6)$,
 $(263, 6)$,
 $(269, 6)$,
 $(277, 6)$ and
 $(283, 6)$.

\section{The Modified Prime Sieve Technique}\label{sec5}

We begin this section with the sieving inequality, as established by Kapetanakis in \cite{10}, which we adjust properly. We use the notation $P(d_1, d_2)$ to denote number of elements $\alpha$ in $S_c^*=\Sc$ such that $\alpha$ is $d_1$-free and $f(\alpha)=a\alpha^2+b\alpha+c$ is $d_2$-free in $\Fm$, where $a\neq 0$ and $b^2\neq 4ac$. Then $P(S_c^*)= P(q^m-1, q^m-1).$

\begin{theorem}\emph{(Sieving inequality).}\label{sieve}
Let $p_1, \ldots, p_s$ be primes dividing $q^m-1$ and $e$ be such that $e\mid q^m-1$, but the $p_i$'s do not divide $e$. Then
\begin{equation}\label{eq5.1}
P(S_c^*)\geq \sum_{i=1}^s P(p_ie, e) + \sum_{i=1}^s P(e, p_ie) -(2s-1)P(e, e).
\end{equation}
\end{theorem}
\begin{theorem}
Under the assumptions of Theorem~\ref{sieve}, define
\begin{equation*}
\delta = 1- 2\sum_{i=1}^s \frac{1}{p_i} \quad \text{and} \quad \Delta = \frac{2s-1}{\delta}+2.
\end{equation*}
If $\delta>0$, then a sufficient condition for the existence of a primitive $S_c^*$ pair in $\Fm$ is
\begin{equation}\label{eq5.2}
\left(\frac{q-1}{\sqrt{q}}\right)^m > \Delta 3.2^m W(e)^2 +(\Delta +1)\left(3. 2^{m-1}-\frac{1}{2} \right) W(e).
\end{equation}
\end{theorem}
\begin{proof}
The equivalent form of \eqref{eq5.1} is as follows.

\begin{equation}\label{eq5.3}
P(S_c^*)\geq \sum_{i=1}^s\{ (P(p_ie, e)-\theta(p_i)P(e,e)) +  (P(e, p_ie) -\theta(p_i)P(e, e)\}+ \delta P(e,e),
\end{equation}
where $\theta(p_i)= 1-\frac{1}{p_i}$.

Upon combining \eqref{eq4.1} and \eqref{eq5.3}, we get
\begin{align*}
P(S_c^*)\geq&   \delta P(e,e)- \left| \sum_{i=1}^s\{ (P(p_ie, e)-\theta(p_i)P(e,e)) +  (P(e, p_ie) -\theta(p_i)P(e, e)\}\right|\\
\geq &  \delta P(e,e)- \sum_{i=1}^s\lbrace \left| (P(p_ie, e)-\theta(p_i)P(e,e))\right| + \left| (P(e, p_ie) -\theta(p_i)P(e, e)\right|\rbrace\\
\geq&  \theta(e)^2 \Big[ \delta \{  (q-1)^m-(3.2^m-1)W(e)q^{m/2} - 3.2^mW(e)^2q^{m/2} \} \\
&- q^{m/2} \sum_{i=1}^s \theta(p_i)\{((2^m-1)+ 2^{m+1})W(e))  + 3.2^{m+1}W(e)^2  \} \Big].
\end{align*}
Then a sufficient condition for $P(S_c^*)>0$ is
\begin{align*}
\delta \{  (q-1)^m-(3.2^m -1)&W(e)q^{m/2} - 3.2^mW(e)^2q^{m/2} \}\\
& - q^{m/2} \sum_{i=1}^s\theta(p_i)\{((3.2^m-1)W(e)  + 3.2^{m+1}W(e)^2)  \} >0,
\end{align*}
that is, 
\begin{equation*}
\left(\frac{q-1}{\sqrt{q}}\right)^m > \Delta 3.2^m W(e)^2 +(\Delta +1)\left(3. 2^{m-1}-\frac{1}{2} \right) W(e).
\end{equation*}
This completes the proof.
\end{proof}


Applying the above sieving condition, we refine the previous pairs and establish~Theorem~\ref{main1}. Table~\ref{table1} includes some pairs $(q,m)$ appearing in the proof of Theorem~\ref{main1}, in which the corresponding test yielded a positive conclusion. In Table~\ref{table1}, the ``RHS'' column denotes  $\Delta 3.2^m W(e)^2 +(\Delta +1)\left(3. 2^{m-1}-\frac{1}{2} \right) W(e)$.


\begin{center}
\begin{longtable}{|c|c|c|c|c|c|c|}
\caption{Table for Theorem~\ref{main1}.} \label{table1} \\
\hline
$(q,m)$ & $e$ & $s$ & $\delta$ &  $\Delta$ &$\left( \frac{q-1}{\sqrt{q}}\right)^m$ & RHS \\ 
\hline
\endfirsthead

\hline
$(q,m)$ & $e$ & $s$ & $\delta$ &  $\Delta$ &$\left( \frac{q-1}{\sqrt{q}}\right)^m$ & RHS \\ 
\hline
\endhead

\hline \multicolumn{7}{|r|}{{Continued on next page}} \\ \hline
\endfoot

\hline 
\endlastfoot

(11, 16) & 6 & 5 & 0.449293  & 22.0315 & $4.66507 \times 10^{7}$ & $3.2158 \times 10^7$\\
\hline
(11, 18) & 6 & 6 & 0.154964  & 72.9842 & $4.24098 \times 10^{8}$ & $4.22485 \times 10^8$\\
\hline
(13, 11) & 6 & 4 & 0.905831  & 9.72771 & $5.55016 \times 10^{5}$ & $4.55016 \times 10^5$\\
\hline
(13, 14) & 6 & 4 & 0.645229  & 12.8489 & $2.04613 \times 10^{7}$ & $4.72964 \times 10^6$\\
\hline
(13, 16) & 6 & 5 & 0.196499  & 47.8018 & $2.26649 \times 10^{8}$ & $6.82753 \times 10^7$\\
\hline
(16, 8) & 3 & 4 & 0.47454 &  16.7511 & $3.91066 \times 10^4$ & $3.07682 \times 10^4$\\
\hline
(16, 10) & 3 & 6 & 0.187206 &  60.7588 & $5.49937 \times 10^5$ & $4.38529 \times 10^5$\\
\hline
(16, 11) & 3 & 6 & 0.481659 &  24.8377 & $2.06226 \times 10^6$ & $3.62191 \times 10^5$\\
\hline
(17, 10) & 2 & 5 & 0.103522 &  88.9381 & $7.74382 \times 10^5$ & $6.4049 \times 10^5$\\
\hline
(17, 12) & 30 & 6 & 0.484935 &  20.5592 & $1.16613 \times 10^7$ & $6.44906\times 10^6$\\
\hline
(19, 10) & 6 & 5 & 0.401853 &  24.3962 & $1.44197 \times 10^6$ & $5.55691\times 10^5$\\
\hline
(19, 12) & 30 & 5 & 0.531041 &  18.9478 & $2.45894 \times 10^7$ & $5.94745 \times 10^6$\\
\hline
(23, 5) & 2 & 2 & 0.81875 &  3.6667 & $2.03139 \times 10^3$ & $9.12674 \times 10^2$\\
\hline
(23, 8) & 6 & 4 & 0.380432 &  20.4001 & $1.96097 \times 10^5$ & $1.16387 \times 10^5$\\
\hline
(25, 7) & 2 & 4 & 0.259811 &  28.9427 & $5.87068 \times 10^4$ & $2.62867 \times 10^4$\\
\hline
(25, 8) & 6 & 4 & 0.721943 &  11.6961 & $2.81793 \times 10^5$ & $6.7383 \times 10^4$\\
\hline
(27, 5) & 2 & 3 & 0.663897 &  7.53129 & $3.13659 \times 10^3$ & $1.77448 \times 10^3$\\
\hline
(27, 6) & 2 & 5 & 0.39848 &  24.5858 & $1.56945\times 10^4$ & $1.11809\times 10^4$\\
\hline
(27, 8) & 2 & 6 & 0.083953 &  133.025 & $3.92945 \times 10^5$ & $2.39015 \times 10^5$\\
\hline
(29, 8) & 6 & 4 & 0.309529 &  24.615 & $5.34161 \times 10^5$ & $1.40116 \times 10^5$\\
\hline
(31, 8) & 6 & 5 & 0.385438 &  25.3501 & $7.10433 \times 10^5$ & $1.44255 \times 10^5$\\
\hline
(32, 6) & 3 & 5 & 0.448664 &  22.0569 & $2.70845 \times 10^4$ & $1.00504 \times 10^4$\\
\hline
(32, 8) & 3 & 6 & 0.187206 &  60.7588 & $8.1338 \times 10^5$ & $9.40188 \times 10^4$\\
\hline
(37, 6) & 6 & 5 & 0.448144 &  21.2249 & $ 4.29744 \times 10^4$ & $ 2.77719 \times 10^4$\\
\hline
(41, 6) & 6 & 4 & 0.309469 &  24.6194 & $5.94304 \times 10^4$ & $3.21115 \times 10^4$\\
\hline
(43, 6) & 6 & 5 & 0.361063 &  26.9264 & $6.90383 \times 10^4$ & $3.82405 \times 10^4$\\
\hline

\end{longtable}
\end{center}




\section{Some conjectures for fields $\Fm$, where $m\leq4$}\label{sec6}


As a conclusion of this work, we build upon the asymptotic results of Theorem~\ref{main2}. In particular, Theorem~\ref{main2} in conjunction with explicit calculations suggests the following conjecture.

\begin{Conjecture}\label{con1}
Let $m\geq 4$, be a positive integer and let $C= \{ A_1, \ldots, A_m \}$ be a set of $\F$-affine hyperplanes of $\Fm$ in general position. Then the set $S_c^*= \Sc$ contains a primitive $S_c^*$ pair $(\alpha, f(\alpha))$ of $\Fm$, where $ax^2+bx+c\in \Fm[x]$ ($a\neq 0$ and $b^2\neq 4ac$), provided one of the following holds:
\begin{enumerate}
\item $q\geq 73,\, m=4$, except when $(q,m)$ is one of $(83, 4)$ and $(89, 4)$. 
\item $q\geq 97, \, m =3$, except when $(q,m)$ is one of $ (103, 3)$, $(107, 3)$, $(151, 3)$, $(191, 3)$ and $(211, 3)$.
\item $ m =2$, and $q$ is large enough.
\end{enumerate}
\end{Conjecture}

\end{document}